\documentclass{amsart}
\usepackage{graphics,amsmath,amsthm,amssymb,tikz}
\usepackage[enableskew]{youngtab}
\usepackage{ytableau}
\usepackage[all,cmtip]{xy}
\usepackage{xcolor}
\usepackage{graphicx}
\usetikzlibrary{matrix}

\newtheorem{theorem}{Theorem}[section]

\newtheorem{proposition}[theorem]{Proposition}
\newtheorem{corollary}[theorem]{Corollary}

\newtheorem{lemma}[theorem]{Lemma}
\newtheorem{example}[theorem]{Example}

\makeatletter
\newsavebox\myboxA
\newsavebox\myboxB
\newlength\mylenA
\newcommand{\la}{\lambda}
\newcommand{\ga}{\gamma}
\newcommand{\dis}{\displaystyle}
\newcommand{\nin}{\noindent}

\newcommand{\s}{\bar{s}}

\newcommand{\g}{\bar{g}}
\newcommand{\Z}{\mathbb{Z}}
\newcommand{\cor}{{\mathcal Cor}}
\newcommand{\quo}{{\mathcal Quo}}

\newcommand*\xoverline[2][0.75]{
    \sbox{\myboxA}{$\m@th#2$}
    \setbox\myboxB\null
        \ht\myboxB=\ht\myboxA
    \dp\myboxB=\dp\myboxA
    \wd\myboxB=#1\wd\myboxA
    \sbox\myboxB{$\m@th\overline{\copy\myboxB}$}
    \setlength\mylenA{\the\wd\myboxA}
    \addtolength\mylenA{-\the\wd\myboxB}
    \ifdim\wd\myboxB<\wd\myboxA
       \rlap{\hskip 0.5\mylenA\usebox\myboxB}{\usebox\myboxA}
    \else
        \hskip -0.5\mylenA\rlap{\usebox\myboxA}{\hskip 0.5\mylenA\usebox\myboxB}
    \fi}
\makeatother

\begin{document}

\title[Simultaneous core partitions with non-trivial common divisor]{Simultaneous core partitions with nontrivial common divisor}

\author[J.-B. Gramain]{Jean-Baptiste Gramain}
\address{Institute of Mathematics, University of Aberdeen, Kings College}

\address{Aberdeen, AB24 3UE, UK}
\address{jbgramain@abdn.ac.uk}

\author[R. Nath]{Rishi Nath}
\address{Department of Mathematics, CUNY-York College, Jamaica, NY 11451, USA,}
\address{rnath@york.cuny.edu}

\author[J. A. Sellers]{James A. Sellers}
\address{Department of Mathematics and Statistics, University of Minnesota Duluth,} 
\address{Duluth, MN 55812, USA,}
\address{jsellers@d.umn.edu}

\date{\bf\today}

\begin{abstract}
A tremendous amount of research has been done in the last two decades on $(s,t)$-core partitions when $s$ and $t$ are relatively prime integers. Here we change perspective slightly and explore properties of $(s,t)$-core and $(\bar{s},\bar{t})$-core partitions for $s$ and $t$ with a nontrivial common divisor $g$.  

We begin by recovering, using the $g$-core and $g$-quotient construction, the generating function for $(s,t)$-core partitions first obtained by D. Aukerman, D. Kane and L. Sze. Then, using a construction developed by the first two authors, we obtain a generating function for the number of $(\bar{s},\bar{t})$-core partitions of $n$. Our approach allows for new results on $t$-cores and self-conjugate $t$-cores that are {\it not} $g$-cores and $\bar{t}$-cores that are {\it not} $\bar{g}$-cores, thus strengthening positivity results of K. Ono and A. Granville, J. Baldwin et. al., and I. Kiming.

We then move to bijections between bar-core partitions and self-conjugate partitions.  We give a new, short proof of a correspondence between self-conjugate $t$-core and $\bar{t}$-core partitions when $t$ is odd and positive first due to J. Yang. Then, using two different lattice-path labelings, one due to B. Ford, H. Mai, and L. Sze, the other to C. Bessenrodt and J. Olsson, we give a bijection between self-conjugate $(s,t)$-core and $(\bar{s},\bar{t})$-core partitions when $s$ and $t$ are odd and coprime. We end this section with a bijection between self-conjugate $(s,t)$-core and $(\bar{s},\bar{t})$-core partitions when $s$ and $t$ are odd and nontrivial $g$ which uses the results stated above.

We end the paper by noting $(s,t)$-core and $(\bar{s}, \bar{t})$-core partitions inherit Ramanujan-type congruences from those of $g$-core and $\bar{g}$-core partitions. 

\end{abstract}

\maketitle

\bigskip


\section{Outline and Background}
We let $s$ and $t$ be positive integers and $g=$gcd$(s,t)$. When $g>1$ we let $s=s't$ and $t=t'g$. We begin, in this section, by reviewing definitions of core partitions, generating functions, and simultaneous core partitions. In Section 2.1, we reprove, in a succinct fashion, Theorems 1.6 and 1.8 and Corollary 1.7 of Aukerman, Kane, and Sze in using known properties of the $g$-core and $g$-quotient. We extend these results to generating functions of self-conjugate $(s,t)$-core partitions and $(\bar{s},\bar{t})$-core partitions in Section 2.2 and Section 2.3.

In Section 3.1 we again reprove Corollary 1.9 in a succinct fashion before strengthening Theorem 1.10 for $g\geq 4$. In Section 3.2 and Section 3.3 we move to analogous results for self-conjugate $(s,t)$-core partitions and $(\bar{s},\bar{t})$-core partitions. In each case, we also set $s=t$ to strengthen existing positivity.

We give a new, shorter proof of Theorem \ref{Yangg} in Section 4.1. In Section 4.2 we provide a bijection between self-conjugate $(s,t)$-core partitions and $(\bar{s},\bar{t})$-partitions when $s$ and $t$ are odd and $g=1$ using the work of both Ford, Mai and Sze and Bessenrodt and Olsson. We combine those results with our work from Section 2.3 to show in Section 4.3 that there is a bijection between self-conjugate $(s,t)$-core partitions and $(\bar{s},\bar{t})$-core partitions when $s,t,g$ are all odd and nontrivial. 

We conclude, in Section 5 by showing how the enumerating functions of $(s,t)$-core partitions when and $(\bar{s},\bar{t})$-core partitions inherit Ramanujan-type congruences from the those of $g$-core and $\bar{g}$-core partitions.

\subsection{partitions, cores and bar-cores}
Let $n$ be a positive integer. A {\it partition} $\lambda$ of a nonnegative integer $n$ is a non-decreasing sequence of positive integers, called {\it parts}, that sum to $n$. We let $|\lambda|=n$ be the {\it size} of $\lambda$. A partition $\lambda$ can be visualized in terms of its {\emph{Young diagram}} $[\lambda]$, a graphic representation of $\lambda$ in which left- and top-aligned rows of boxes correspond sequentially to the parts in the partition. [When the context is clear we will abuse notation and let $\lambda$ refer to both the Young diagram and the partition.] A {\it hook} ${\mathcal H}_{\iota \gamma}$ of $[\lambda]$ with {\it corner} $(\iota, \gamma)$, using matrix notation, is the set of boxes in the same row and to the right of $(\iota, \gamma)$, in the same column and below $(\iota, \gamma)$, and $(\iota, \gamma)$ itself. Then ${\mathcal H}=\{{\mathcal H}_{\iota,\gamma} \, | \, \iota, \gamma \geq 1\}$ is the {\it hook set} of $\lambda$. The set $\{{\mathcal H}_{\iota1} \, | \, \iota \geq 1\}$ consists of the {\it first column hooks} of [$\lambda$]. The {\it hook length} $h_{\iota \gamma}$ (or (or $|{\mathcal H}_{\iota\gamma}|$) of ${\mathcal H}_{\iota \gamma}$ is the number of boxes in the hook. The set $\{h_{\iota \gamma}\}$ will be called the {\it set of hook lengths} of $\lambda$. In particular, the set $\{h_{\iota1}\}$ of {\it first column hook lengths} of $\lambda$ completely determine the partition. The Young diagram $[\lambda^{\vee}]$ of the {\it conjugate partition} $\lambda^{\vee}$ of $\lambda$ is obtained by exchanging rows and columns of the Young diagram of $[\lambda]$. We say $\lambda$ is a {\it self-conjugate} partition if $\lambda=\lambda^{\vee}$. A self-conjugate partition is completely determined by its {\it set of diagonal hook lengths} $\Delta(\lambda)=\{h_{ii}\}$. 

Let $\{h_{\iota\gamma} \, | \, \iota, \gamma \geq 1\}$ be the {\it multiset of hook lengths of} $\lambda$ (including repetitions). A {\it t-hook} is a hook of length $t$. For fixed $t$, partitions whose multiset of hook lengths do not contain any hooks of length $t$ are known as {\it t-core partitions}. The {\it positivity} of $t$-core partitions (see Theorem \ref{Ono} below) was proved in several steps (see \cite{EM},  \cite{OG},  \cite{KO}, \cite{On}).

\begin{theorem}[Positivity] \label{Ono} Every non-negative integer $n$ has at least one $t$-core partition for $t\geq 4.$ 
\end{theorem}

{\it Bar partitions} are partitions with distinct parts. Given a bar partition $\lambda$, one obtains its {\it shifted Young diagram} $S(\lambda)$ by shifting the $i$th row in the Young diagram $i-1$ boxes to the right. The {\it shift-symmetric diagram}  $SS(\lambda)$ is then obtained by adjoining the parts of $\lambda$ as columns to $S(\lambda)$ in the following way: a column consisting of $\lambda_i$ boxes is attached one position to the left of the $(i,i)$ position in $S(\lambda)$. Then we can define a {\it bar} and {\it bar length}, analogous to hook and hook length, in the following way: a bar is associated to each box in $\lambda$, and its bar length is the hook length of the corresponding box of $S(\lambda)$ embedded in the shift-symmetric Young diagram $SS(\lambda)$ (for details, see \cite{Be}). 
\begin{example}
Consider the Young diagram of the bar partition $\lambda=(5,3,1)$ .
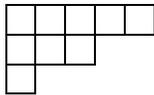
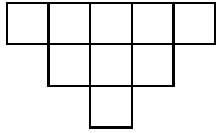
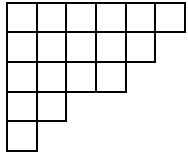
\[ \yng(5,3,1)
\]
We have $S(\lambda)$
\[
\ydiagram{5,1+3,2+1}
\]
and $SS(\lambda)$
\[
\yng(6,5,4,2,1)
\]
Now consider the hook lengths of the boxes in $S(\lambda)$ embedded in $SS(\lambda)$.
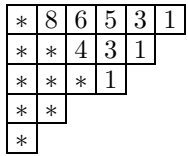
\[
\young(*86531,**431,***1,**,*)
\]
Returning to $S(\lambda)$ we obtain the multiset of bar lengths $\{8,6,5,4,3,3,1,1,1\}$ associated to $\lambda$.
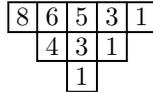
\[ \young(86531,:431,::1)
\]
\end{example}
Partitions whose multiset of bar lengths do not contain bars of length $t$ are called $t${\it-bar-core}, or $\bar{t}$-core partitions. J. Baldwin, M. Depweg, B. Ford, A. Kunin, and L. Sze \cite{BDFS} proved a version of the positivity conjecture for self-conjugate core partitions; I. Kiming \cite{K2} proved a version of positivity for bar-core partitions. [Kiming's theorem was originally proven for $p$ prime (relevant for the so-called $p$-defect zero spin blocks), but the argument goes through without change for all odd $t\geq 7$ \cite{Ga}. We will use the more general statement.]

\begin{theorem} \label{fordsze} Every non-negative integer $n$ has at least one self-conjugate $t$-core partition for $t=8$ or $t\geq 10$.
\end{theorem}

\begin{theorem} \label{kimingtime} Every non-negative integer $n$ has a $\bar{t}$-core partition for $t\geq 7$ and odd.
\end{theorem}

The next theorem is a recent result of J. Yang which connects self-conjugate $t$-core partitions and $\bar{t}$-core partitions.
\begin{theorem} \cite[Theorem 1.1]{Yang} \label{Yangg} Let $t$ be odd. There is a correspondence $\zeta$ between the self-conjugate $t$-cores and $\bar{t}$-core partitions.
\end{theorem}
\subsection{Generating functions}
Euler determined the generating function $P$ for the number $p(n)$ of partitions of $n$: 
\[
P(x):=\sum^{\infty}_{n=0}p(n)x^n=\prod_{n=1}^\infty\frac{1}{1-x^n}
\]
The generating function $F_t$ for the number $f_t(n)$ of $t$-core partitions of $n$ was first obtained by J. B. Olsson \cite{O0}.
$$F_{{t}}(x)=\sum^{\infty}_{n=0}f_t(n)x^n=\prod^{\infty}_{n=1}\frac{(1-x^{tn})^t}{(1-x^n)}$$
F. Garvan, D. Kim, and D. Stanton (Eqs. 7.1(a) and (b) in \cite{GKS}) first described the generating function $F^*_t$ for the number $f^*_t(n)$ of self-conjugate $t$-core partitions of $n$. We use a formulation that appears in \cite{BDFS}.
\[   F^*_{t}(x)=\sum^{\infty}_{n=0}f^*_t(n)x^n=\left\{
\begin{array}{ll}
      \prod^{\infty}_{n=1}(1-x^{2tn})^{\frac{t}{2}}(1+x^{2n-1}) & \mbox{if $t$ is even,}\\\\
      \prod^{\infty}_{n=1}(1-x^{2tn})^{\frac{t-1}{2}}(\frac{1+x^{2n-1}}{1+x^{t(2n-1)}}) & \mbox{if $t$ is odd.} \\
\end{array} 
\right. \]
The generating function $F_{\bar{t}}$ for the number $f_{\bar{t}}(n)$ of $\bar{t}$-core partitions of $n$ was also found by Olsson \cite{O3}.
$$F_{\bar{t}}(x)=\sum^{\infty}_{n=0}f_{\bar{t}}(n)x^n=\dis \prod_{n=1}^{\infty}\frac{(1-x^{2n})(1-x^{tn})^{\frac{t+1}{2}}}{(1-x^n)(1-x^{2tn})}$$

\subsection{$(s,t)$-core partitions, $g=1$.}
A {\it simultaneous core partition} is a partition whose set of hook lengths avoids a specified subset of positive integers. An $(s,t)$-core partition is both an $s$-core and a $t$-core partition, where $s,t>1$.\\
 
J. Anderson (Section 3, \cite{An}) describes, when $g=1$, the possible first-column hook lengths of an $(s,t)$-core partition.  Recall a lattice path is {\it monotonic} if it only moves one position up or one position to the right at every step. Then, in particular, Anderson's results imply a correspondence between $(s,t)$-cores and a family of monotonic paths that stay above the diagonal in a certain $s\times t$ lattice, constructed in the following way: The left-and-topmost box, with upper-left corner labeled $(0,s)$, contains the value $st-s-t$. One horizontal position right from $(0,s)$ is the box whose upper left corner is labeled $(1,s)$, and the value in that box is $st-2s-t$ (a reduction by $s$). One vertical position down from $(0,s)$ is the box with label $(0,s-1)$; the value inside is $st-s-2t$. Complete the lattice coordinates with their values in this fashion. 

\begin{example} Let $s=7$ and $t=11$. Anderson's lattice is shown below. The monotonic path $\pi$ appearing in the lattice corresponds to the $(7,11)$-core partition $\lambda_{\pi}$ whose first-column hook lengths $\{h_{\iota1}\}_{1\leq \iota\leq 5}=\{13,10,9,8,6,5,3,2,1\}$ are the values trapped between the (solid) path and the (dashed) $\pm$border. This corresponds to the partition $\lambda_{\pi}=(5,3^3,2^2,1^3).$
\end{example}

\smallskip

\begin{center}
\scalebox{0.25}{\begin{tikzpicture}
\draw (0cm,12cm) node [rectangle, minimum size=2cm, inner sep=0pt, draw, anchor=south west] {\Huge$59$};
\draw (2cm,12cm) node [rectangle, minimum size=2cm, inner sep=0pt, draw, anchor=south west] {\Huge$52$};
\draw (4cm,12cm) node [rectangle, minimum size=2cm, inner sep=0pt, draw, anchor=south west] {\Huge$45$};
\draw (6cm,12cm) node [rectangle, minimum size=2cm, inner sep=0pt, draw, anchor=south west] {\Huge$38$};
\draw (8cm,12cm) node [rectangle, minimum size=2cm, inner sep=0pt, draw, anchor=south west] {\Huge$31$};
\draw (10cm,12cm) node [rectangle, minimum size=2cm, inner sep=0pt, draw, anchor=south west] {\Huge$24$};
\draw (12cm,12cm) node [rectangle, minimum size=2cm, inner sep=0pt, draw, anchor=south west] {\Huge$17$};
\draw (14cm,12cm) node [rectangle, minimum size=2cm, inner sep=0pt, draw, anchor=south west] {\Huge$10$};
\draw (16cm,12cm) node [rectangle, minimum size=2cm, inner sep=0pt, draw, anchor=south west] {\Huge$3$};
\draw (18cm,12cm) node [rectangle, minimum size=2cm, inner sep=0pt, draw, anchor=south west] {\Huge$-4$};
\draw (20cm,12cm) node [rectangle, minimum size=2cm, inner sep=0pt, draw, anchor=south west] {\Huge$-11$};

\draw (0cm,10cm) node [rectangle, minimum size=2cm, inner sep=0pt, draw, anchor=south west] {\Huge$48$};
\draw (2cm,10cm) node [rectangle, minimum size=2cm, inner sep=0pt, draw, anchor=south west] {\Huge$41$};
\draw (4cm,10cm) node [rectangle, minimum size=2cm, inner sep=0pt, draw, anchor=south west] {\Huge$34$};
\draw (6cm,10cm) node [rectangle, minimum size=2cm, inner sep=0pt, draw, anchor=south west] {\Huge$27$};
\draw (8cm,10cm) node [rectangle, minimum size=2cm, inner sep=0pt, draw, anchor=south west] {\Huge$20$};
\draw (10cm,10cm) node [rectangle, minimum size=2cm, inner sep=0pt, draw, anchor=south west] {\Huge$13$};
\draw (12cm,10cm) node [rectangle, minimum size=2cm, inner sep=0pt, draw, anchor=south west] {\Huge$6$};
\draw (14cm,10cm) node [rectangle, minimum size=2cm, inner sep=0pt, draw, anchor=south west] {\Huge$-1$};
\draw (16cm,10cm) node [rectangle, minimum size=2cm, inner sep=0pt, draw, anchor=south west] {\Huge$-8$};
\draw (18cm,10cm) node [rectangle, minimum size=2cm, inner sep=0pt, draw, anchor=south west] {\Huge$-15$};
\draw (20cm,10cm) node [rectangle, minimum size=2cm, inner sep=0pt, draw, anchor=south west] {\Huge$-22$};

\draw (0cm,8cm) node [rectangle, minimum size=2cm, inner sep=0pt, draw, anchor=south west] {\Huge$37$};
\draw (2cm,8cm) node [rectangle, minimum size=2cm, inner sep=0pt, draw, anchor=south west] {\Huge$30$};
\draw (4cm,8cm) node [rectangle, minimum size=2cm, inner sep=0pt, draw, anchor=south west] {\Huge$23$};
\draw (6cm,8cm) node [rectangle, minimum size=2cm, inner sep=0pt, draw, anchor=south west] {\Huge$16$};
\draw (8cm,8cm) node [rectangle, minimum size=2cm, inner sep=0pt, draw, anchor=south west] {\Huge$9$};
\draw (10cm,8cm) node [rectangle, minimum size=2cm, inner sep=0pt, draw, anchor=south west] {\Huge$2$};
\draw (12cm,8cm) node [rectangle, minimum size=2cm, inner sep=0pt, draw, anchor=south west] {\Huge$-5$};
\draw (14cm,8cm) node [rectangle, minimum size=2cm, inner sep=0pt, draw, anchor=south west] {\Huge$-12$};
\draw (16cm,8cm) node [rectangle, minimum size=2cm, inner sep=0pt, draw, anchor=south west] {\Huge$-19$};
\draw (18cm,8cm) node [rectangle, minimum size=2cm, inner sep=0pt, draw, anchor=south west] {\Huge$-26$};
\draw (20cm,8cm) node [rectangle, minimum size=2cm, inner sep=0pt, draw, anchor=south west] {\Huge$-33$};

\draw (0cm, 6cm) node [rectangle, minimum size=2cm, inner sep=0pt, draw, anchor=south west] {\Huge$26$};
\draw (2cm, 6cm) node [rectangle, minimum size=2cm, inner sep=0pt, draw, anchor=south west] {\Huge$19$};
\draw (4cm, 6cm) node [rectangle, minimum size=2cm, inner sep=0pt, draw, anchor=south west] {\Huge$12$};
\draw (6cm, 6cm) node [rectangle, minimum size=2cm, inner sep=0pt, draw, anchor=south west] {\Huge$5$};
\draw (8cm, 6cm) node [rectangle, minimum size=2cm, inner sep=0pt, draw, anchor=south west] {\Huge$-2$};
\draw (10cm, 6cm) node [rectangle, minimum size=2cm, inner sep=0pt, draw, anchor=south west] {\Huge$-9$};
\draw (12cm, 6cm) node [rectangle, minimum size=2cm, inner sep=0pt, draw, anchor=south west] {\Huge$-16$};
\draw (14cm, 6cm) node [rectangle, minimum size=2cm, inner sep=0pt, draw, anchor=south west] {\Huge$-23$};
\draw (16cm, 6cm) node [rectangle, minimum size=2cm, inner sep=0pt, draw, anchor=south west] {\Huge$-30$};
\draw (18cm, 6cm) node [rectangle, minimum size=2cm, inner sep=0pt, draw, anchor=south west] {\Huge$-37$};
\draw (20cm, 6cm) node [rectangle, minimum size=2cm, inner sep=0pt, draw, anchor=south west] {\Huge$-44$};

\draw (0cm, 4cm) node [rectangle, minimum size=2cm, inner sep=0pt, draw, anchor=south west] {\Huge$15$};
\draw (2cm, 4cm) node [rectangle, minimum size=2cm, inner sep=0pt, draw, anchor=south west] {\Huge$8$};
\draw (4cm, 4cm) node [rectangle, minimum size=2cm, inner sep=0pt, draw, anchor=south west] {\Huge$1$};
\draw (6cm, 4cm) node [rectangle, minimum size=2cm, inner sep=0pt, draw, anchor=south west] {\Huge$-6$};
\draw (8cm, 4cm) node [rectangle, minimum size=2cm, inner sep=0pt, draw, anchor=south west] {\Huge$-13$};
\draw (10cm, 4cm) node [rectangle, minimum size=2cm, inner sep=0pt, draw, anchor=south west] {\Huge$-20$};
\draw (12cm, 4cm) node [rectangle, minimum size=2cm, inner sep=0pt, draw, anchor=south west] {\Huge$-27$};
\draw (14cm, 4cm) node [rectangle, minimum size=2cm, inner sep=0pt, draw, anchor=south west] {\Huge$-34$};
\draw (16cm, 4cm) node [rectangle, minimum size=2cm, inner sep=0pt, draw, anchor=south west] {\Huge$-41$};
\draw (18cm, 4cm) node [rectangle, minimum size=2cm, inner sep=0pt, draw, anchor=south west] {\Huge$-48$};
\draw (20cm, 4cm) node [rectangle, minimum size=2cm, inner sep=0pt, draw, anchor=south west] {\Huge$-55$};

\draw (0cm, 2cm) node [rectangle, minimum size=2cm, inner sep=0pt, draw, anchor=south west] {\Huge$4$};
\draw (2cm, 2cm) node [rectangle, minimum size=2cm, inner sep=0pt, draw, anchor=south west] {\Huge$-3$};
\draw (4cm, 2cm) node [rectangle, minimum size=2cm, inner sep=0pt, draw, anchor=south west] {\Huge$-10$};
\draw (6cm, 2cm) node [rectangle, minimum size=2cm, inner sep=0pt, draw, anchor=south west] {\Huge$-17$};
\draw (8cm, 2cm) node [rectangle, minimum size=2cm, inner sep=0pt, draw, anchor=south west] {\Huge$-24$};
\draw (10cm, 2cm) node [rectangle, minimum size=2cm, inner sep=0pt, draw, anchor=south west] {\Huge$-31$};
\draw (12cm, 2cm) node [rectangle, minimum size=2cm, inner sep=0pt, draw, anchor=south west] {\Huge$-38$};
\draw (14cm, 2cm) node [rectangle, minimum size=2cm, inner sep=0pt, draw, anchor=south west] {\Huge$-45$};
\draw (16cm, 2cm) node [rectangle, minimum size=2cm, inner sep=0pt, draw, anchor=south west] {\Huge$-52$};
\draw (18cm, 2cm) node [rectangle, minimum size=2cm, inner sep=0pt, draw, anchor=south west] {\Huge$-59$};
\draw (20cm, 2cm) node [rectangle, minimum size=2cm, inner sep=0pt, draw, anchor=south west] {\Huge$-66$};

\draw (0cm, 0cm) node [rectangle, minimum size=2cm, inner sep=0pt, draw, anchor=south west] {\Huge$-7$};
\draw (2cm, 0cm) node [rectangle, minimum size=2cm, inner sep=0pt, draw, anchor=south west] {\Huge$-14$};
\draw (4cm, 0cm) node [rectangle, minimum size=2cm, inner sep=0pt, draw, anchor=south west] {\Huge$-21$};
\draw (6cm, 0cm) node [rectangle, minimum size=2cm, inner sep=0pt, draw, anchor=south west] {\Huge$-28$};
\draw (8cm, 0cm) node [rectangle, minimum size=2cm, inner sep=0pt, draw, anchor=south west] {\Huge$-35$};
\draw (10cm, 0cm) node [rectangle, minimum size=2cm, inner sep=0pt, draw, anchor=south west] {\Huge$-42$};
\draw (12cm, 0cm) node [rectangle, minimum size=2cm, inner sep=0pt, draw, anchor=south west] {\Huge$-49$};
\draw (14cm, 0cm) node [rectangle, minimum size=2cm, inner sep=0pt, draw, anchor=south west] {\Huge$-56$};
\draw (16cm, 0cm) node [rectangle, minimum size=2cm, inner sep=0pt, draw, anchor=south west] {\Huge$-63$};
\draw (18cm, 0cm) node [rectangle, minimum size=2cm, inner sep=0pt, draw, anchor=south west] {\Huge$-70$};
\draw (20cm, 0cm) node [rectangle, minimum size=2cm, inner sep=0pt, draw, anchor=south west] {\Huge$-77$};

\draw[line width=8pt](0,0)--(0,2)--(2,2)--(2,6)--(4,6)--(6,6)--(6,8)--(8,8)--(8,10)--(10,10)--(10,12)--(14,12)--(14,14)--(22,14);
\draw[line width=12pt,dashed](0,0)--(0,2)--(2,2)--(2,4)--(6,4)--(6,6)--(8,6)--(8,8)--(12,8)--(12,10)--(14,10)--(14,12)--(18,12)--(18,14)--(22,14);
\draw[line width=5pt](0,0)--(22,14);

\filldraw[black] (0,0) circle (0.25cm);

\filldraw[black] (22,14) circle (0.25cm);
\end{tikzpicture}}
\end{center}

Anderson uses this to prove that the number of $(s,t)$-core partitions is $\frac{1}{(s+t)}\binom{s+t}{t}$.  J. B. Olsson and D. Stanton \cite{OS} showed that the largest size of such an $(s,t)$-core partition is $\frac{(s^2-1)(t^2-1)}{24}$, verifying a conjecture of B. Kane \cite{Bk}. 

\subsection{$(s,t)^*$-core and $(\bar{s},\bar{t})$-core partitions, $g=1$.}
Let $s,t>1$ be odd. We use the notation $(s,t)^*$-core partition to indicate an $(s,t)$-core partition that is also self-conjugate. B. Ford, H. Mai and L. Sze \cite{fms} construct a bijection between $(s,t)^*$-core partitions and monotonic paths in a certain ${\left\lfloor\frac{s}{2}\right\rfloor}\times {\left\lfloor\frac{t}{2}\right\rfloor}$ lattice which we call the {\it diagonal hooks diagram}. As a consequence, they find that, when $g=1$, the number of self-conjugate $(s,t)$-core partitions is $${\binom{{\left\lfloor\frac{s}{2}\right\rfloor}+{\left\lfloor\frac{t}{2}\right\rfloor}}{{\left\lfloor\frac{s}{2}\right\rfloor}}}.$$

A bar-partition is said to be an $(\s,\bar{t})$-core if it is both an $\s$-core and a $\bar{t}$-core. In the case that $s,t>1$ are odd and $g=1$, C. Bessenrodt and J. B. Olsson \cite{BO} construct a bijection between $(\bar{s},\bar{t})$-core partitions and monotonic paths in a $\frac{s-1}{2}\times\frac{t-1}{2}$ array, called the {\it Yin-Yang diagram}. As a consequence they find the number of $(\bar{s},\bar{t})$-core partitions and the size of the $(\bar{s},\bar{t})$-core which ``contains" all others.

We outline both the Ford-Mai-Sze and the Bessenrodt-Olsson constructions in Section 4.2.

\subsection{$(s,t)$-core and $(s,t)^*$-core partitions, $g>1$.}
When $g>1$, there are infinitely many $(s,t)$-cores, self-conjugate $(s,t)$-cores and $(\bar{s},\bar{t})$-cores, so there is no largest such partition. However, some properties of $(s,t)$-cores for $g=1$ remain true when $g>1$ (e.g. \cite{GN12} and \cite{Na07} generalize results in \cite{O4}). D. Aukerman, B. Kane and L. Sze first characterized an $(s,t)$-core partition $\la$ in terms of its $g$-core ${\mathcal Cor}_g(\la)$ and $g$-quotient ${\mathcal Quo}_g(\la)$, which we define formally in Section 2.1.

\begin{theorem} \cite[Theorem 1.3]{AKS}\label{stcore}
Let $\la$ be any integer partition, $s=s'g$ and $t=t'g$, and let ${\mathcal Quo}_g(\la)=(\la_1, \, \ldots , \, \la_g)$ be the $g$-quotient of $\la$. Then $\la$ is an $(s,t)$-core if and only if $\la_1, \, \ldots , \, \la_g$ are $(s',t')$-cores.
\end{theorem}

They obtain, as a corollary, a generating function $\Psi_{s,t}$ for the number of $(s,t)$-core partitions of $n$. 
\begin{corollary} \cite[Corollary 4.3]{AKS} 
\label{genfunstcore} 
With the above notation, we have
$$\Psi_{s,t}(x)= \Psi_{s',t'}(x^g)^g  F_g(x) =\Psi_{s',t'}(x^g)^g  \prod_{n=1}^{\infty}\frac{(1-x^{gn})^g}{1-x^n},$$
where $F_g$ is the generating function for the number of $g$-core partitions.
\end{corollary}
They also find the generating function for the number of self-conjugate $(s,t)$-cores when $g>1$.

\begin{theorem} \cite[Theorem 7.1]{AKS}
\label{genfunselfstcore}
The generating function $\Psi^*_{s,t}$ for the number of self-conjugate $(s,t)$-cores with $g>1$ is given by
$$\Psi^*_{s,t} (x)= \left\{ \begin{array}{ll}F^*_{g}(x)(\Psi_{s',t'}(x^{2g}))^{\frac{g}{2}} &  \mbox{if g is even,}\\
F^*_{g}(x)(\Psi_{s',t'}(x^{2g}))^{\frac{g-1}{2}}\Psi^*_{s',t'}(x^g) &  \mbox{if g is odd.}
\end{array}
\right.$$
\end{theorem}

Aukerman, Kane, and Sze also obtain a result on the number of $(s,t)$-core partitions which are not $g$-cores, and the existence of a $t$-core partition of $n$ that is not a $g$-core.

\begin{theorem}\cite[Corollary 4.4]{AKS}
\label{infinitelymanystcores}
If $g>1$, $s'>1$ and $t'>1$, then there are infinitely many $(s'g,t'g)$-cores which are not $g$-cores.
\end{theorem}

\begin{theorem}\label{yeahtnog} \cite[Corollary 1.2]{AKS} Let $g \geq 4$ be an integer. Then, for any integers $n \geq g$ and $t'>1$, there exists a partition of $n$ which is a $t'g$-core but not a $g$-core.
\end{theorem}
Aukerman, Kane and Sze do not consider $(\bar{s},\bar{t})$-core partitions with $g>1$; we consider them here using a construction of the first two authors \cite{GN12}.

[Note: All enumerating functions, generating functions, and bijections referred to in this paper are listed in an Index of Functions in Section 6.]

\section{Generating functions}

\subsection{$(s,t)$-cores with $g>1$}

We start with some relevant background. For more details we refer the reader to \cite[Section 3]{O3}. Recall a {\it t-core}, or {\it t-core partition}, is one whose multiset of hook lengths avoids $t$, or has no $t$-hooks. More generally, the {\it t-core of} $\lambda$, denoted ${\mathcal Cor}_t{(\la)}$, is obtained by removing successive $t$-hooks until no further $t$-hooks can be removed.  If a partition has $w$ such $t$-hooks, then the {\it t-quotient} ${\mathcal Quo}_t(\la)$ is an $t$-tuple $(\la_1, \, \ldots , \, \la_t)$ of partitions whose sizes add up to $w$. The integer $w$ is then called the {\emph{$t$-weight}} of $\la$, and ${\mathcal Cor}_t(\la)$ is a partition of $n-wt$. The fundamental property of ${\mathcal Quo}_t(\la)$ is that it contains all the information about the hooks of length divisible by $t$ in $\la.$ The following lemma and proposition are well-known.

\begin{lemma}\label{2.1}
If $\lambda$ is a partition of $n$, and if $t$ is any positive integer, then $\lambda$ is completely and uniquely determined by its {\emph{$t$-core}} ${\mathcal Cor}_t{(\la)}$ and {\emph{$t$-quotient}} ${\mathcal Quo}_t(\la)$.
\end{lemma} 

\begin{proposition} \label{hookbij}
For any integer $k \geq 1$, there is a canonical bijection between the set of hooks of length $kt$ in $\la$ and the set of hooks of length $k$ in ${\mathcal Quo}_t(\la)$, where a hook in the $t$-quotient is simply a hook in any of $\la_1, \, \ldots , \, \la_t$.
\end{proposition}

We now provide new, succinct proofs of Theorem \ref{stcore} and Corollary \ref{genfunstcore}. We note that the approach here provides the blueprint for many of the results that follow.
\begin{proof}[{\bf{Proof of Theorem \ref{stcore}}}]
Recall that $\la$ is an $s$-core if and only if $\la$ has no hook of length divisible by $s=s'g$. Now, by Proposition \ref{hookbij}, there is a bijection between the set of hooks of length $ks=ks'g$ in $\la$ and the set of hooks of length $ks'$ in ${\mathcal Quo}_g(\la)$ for any $k \geq 1$. Hence $\la$ is an $s$-core if and only if $\la_1, \, \ldots , \, \la_g$ are $s'$-cores. Similarly, $\la$ is a $t$-core if and only if $\la_1, \, \ldots , \, \la_g$ are $t'$-cores. The result follows.
\end{proof}

We write $Q^{(s',t')}_g(w)$ for the number of $g$-tuples of $(s',t')$-cores of weight $w$, and $\psi_{s,t}(n)$ for the number of $(s,t)$-cores of an integer $n$. 

\begin{proof}[{\bf{Proof of Corollary \ref{genfunstcore}}}]
The proof is similar to that of \cite[Proposition (9.4)(iii)]{O3}, and follows from a counting argument using Theorem \ref{stcore}. By Lemma \ref{2.1}, each $(s,t)$-core partition $\la$ of $n$ is completely and uniquely determined by its $g$-core ${\mathcal Cor}_g(\la)$ and $g$-quotient ${\mathcal Quo}_g(\la)$. And, by Theorem \ref{stcore}, if $\la$ has $g$-weight $w$, then ${\mathcal Quo}_g(\la)$ can be any $g$-tuple of $(s',t')$-cores of weight $w$.  We have
\begin{equation}\label{jade}
\psi_{s,t}(n)=\dis \sum_{w\geq 0} Q^{(s',t')}_g(w)f_g(n-gw).
\end{equation}
Recall that $F_g(x)$ is the generating function for the number of $g$-cores of $n$. Then, using an argument similar to Proposition 9.14(ii) in \cite{O3},
Equation (\ref{jade}) can be transformed into the generating function
$$\Psi_{s,t}(x)= \Psi_{s',t'}(x^g)^g  F_g(x),$$
as claimed. 
\end{proof} 

\subsection{$(s,t)^*$-cores with $g>1$}
We consider self-conjugate $(s,t)$-core partitions, or $(s,t)^*${\it -core} partitions. Our description of the core and quotient of a self-conjugate $(s,t)$-core partition is a consequence of the following lemma which appears in \cite{HN}.

\begin{lemma}\label{nec0} Let $t$ be any positive integer, $\lambda$ be a partition, and ${\mathcal Quo}_t(\la)=(\la_0, \ldots , \la_{t-1})$. Then $\lambda$ is self-conjugate if and only if ${\mathcal Cor}_t(\la)$ is self-conjugate and $\lambda^{\vee}_{i}=\lambda_{t-i-1}$ for all $0 \leq i \leq t-1$.
\end{lemma}

\begin{lemma}\label{nec} A partition $\lambda$ is an self-conjugate $(s,t)$-core partition if and only if ${\mathcal Cor}_g(\lambda)$ is self-conjugate and each part of ${\mathcal Quo}_g(\lambda)=(\lambda_0,\cdots,\lambda_{g-1})$ is an $(s',t')$-core such that $\lambda_i=\lambda^{\vee}_{g-i-1}$.
\end{lemma}

\begin{proof} This follows from Theorem \ref{stcore} and Lemma \ref{nec0}.
\end{proof}
Let $\psi^*_{s,t}(n)$ be the number of self-conjugate $(s,t)$-core partitions of $n$.
\begin{proof}[{\bf{Proof of Theorem \ref{genfunselfstcore}}}]

If $\la$ is any self-conjugate $(s,t)$-core of $n$ of $g$-weight $v$, then ${\mathcal Cor}_g(\lambda)$ can be any self-conjugate $g$-core of $n-vg$.\\

\nin
If $g$ is even, then ${\mathcal Quo}_g(\lambda)$ has the form $(\la_0, \ldots , \la_{\frac{g}{2}-1}, \la^{\vee}_{\frac{g}{2}-1}, \ldots , \la^{\vee}_0)$, and $(\la_0, \ldots , \la_{\frac{g}{2}-1})$ can be any $\frac{g}{2}$-tuple of $(s',t')$-cores whose sizes add up to $w$ such that $v=2w$. The number of such $\la$'s is thus
\begin{equation}
\psi^*_{s,t}(n)=\dis \sum^{\frac{n}{2g}}_{w\geq 0} Q^{(s',t')}_{ \frac{g}{2}}(w) f^*_g(n-2wg)
\end{equation}
in this case. [Note: Here $\frac{n}{g}\geq2w_1+w_2$.]\\ 

\nin
If, on the other hand, $g$ is odd, then ${\mathcal Quo}_g(\lambda)=(\la_0, \ldots , \la_{\frac{g-3}{2}}, \la_{\frac{g-1}{2}}, \la^{\vee}_{\frac{g-3}{2}}, \ldots , \la^{\vee}_0)$, where $(\la_0, \ldots , \la_{\frac{g-3}{2}})$ can be any $\frac{g-1}{2}$-tuple of $(s',t')$-cores whose sizes add up to $w_1$ say, and $\la_{\frac{g-1}{2}}$ can be any self-conjugate $(s',t')$-core of $w_2$, with $v=2w_1+w_2$. The number of such $\la$ (over all such $g$-weights $\nu$) is thus
\begin{equation}
\psi^*_{s,t}(n)= \dis \sum^{\frac{n}{g}}_{2w_1+w_2\geq 0} Q^{(s',t')}_{\frac{g-1}{2}}(w_1) \psi^*_{s',t'}(w_2) f^*_g(n-(2w_1+w_2)g)
\end{equation}
in this case.\\

\nin
These yield the desired generating functions.

\end{proof}

\subsection{$(\bar{s},\bar{t})$-cores with $g>1$}

In this section, we state and prove new bar-analogues of Theorem \ref{stcore} and Corollary \ref{genfunstcore}.

\medskip

We start with some background on bar partitions. For more details, we refer the reader to \cite[Section 4]{O3}.  Recall a $\bar{t}$-core, or $\bar{t}$-core partition, is any bar partition whose multiset of bar lengths avoids $t$, or contains no bars of length $t$. More generally, $\cor_{\bar{t}}(\la)$ the $\bar{t}$-core of $\la$, is obtained by removing successive bars of length $t$ from $\la$ until no bars of length $t$ remain. If there are $w$ such bars of length $t$, then the $\bar{t}$-quotient ${\mathcal Quo}_{\bar{t}}(\la)$ is an $\frac{t+1}{2}$-tuple $(\la_0, \, \la_1, \, \ldots , \, \la_{\frac{t-1}{2}})$ of partitions whose sizes add up to $w$, where $\la_0$ is itself a bar-partition. The integer $w$ is then called the {\emph{$\bar{t}$-weight}} of $\la$, and ${\mathcal Cor}_{\bar{t}}(\la)$ is a bar-partition of $n-wt$. The $\bar{t}$-quotient ${\mathcal Quo}_{\bar{t}}(\la)$ is said to have $t$-weight $w$.

\begin{lemma} \cite[Proposition (4.2)]{O3}
If $\lambda$ is a bar-partition, and if $t$ is any odd positive integer, then $\lambda$ is completely and uniquely determined by ${\mathcal Cor}_{\bar{t}}(\lambda)$ and ${\mathcal Quo}_{\bar{t}}(\la)$.
\end{lemma}

In analogy with Proposition \ref{hookbij}, we have the following:

\begin{proposition} \label{barbij} \cite[Theorem (4.3)]{O3}
If $t$ is an odd positive integer and $\la$ is a bar-partition, then, for any integer $k \geq 1$, there is a canonical bijection between the set of bars of length $kt$ in $\la$ and the set of bars of length $k$ in ${\mathcal Quo}_{\bar{t}}(\la)$, where a bar in ${\mathcal Quo}_{\bar{t}}(\la)$ is defined to be a bar in $\la_0$, or a hook in any of $\la_1, \, \ldots , \, \la_{\frac{t-1}{2}}$.
\end{proposition}

Now consider odd integers $s=s'g$ and $t=t'g$, where $g>1$. 
We denote by $\Psi_{\s,\bar{t}}$ the generating function for the number of $(\s,\bar{t})$-core partitions. We now have the following bar-analogues of Theorem \ref{stcore} and Corollary \ref{genfunstcore}. 

\begin{theorem} \label{stbarcore}
Let $\la$ be any bar-partition, and let ${\mathcal Quo}_{\bar{g}}(\la)=(\la_0, \, \la_1, \, \ldots , \, \la_{\frac{g-1}{2}})$ be the $\g$-quotient of $\la$. Then $\la$ is an $(\s,\bar{t})$-core if and only if 
\begin{enumerate}
\item $\la_0$ is an $(\bar{s}', \bar{t}')$-core and\\
\item $\la_1, \, \ldots , \, \la_{\frac{g-1}{2}}$ are $(s',t')$-cores.
\end{enumerate}
\end{theorem}

\begin{proof}
Note that $\la$ is an $\s$-core if and only if $\la$ has no bar of length divisible by $s=s'g$. Now, by Proposition \ref{barbij}, for any $k \geq 1$, there is a bijection between the set of bars of length $ks=kgs'$ in $\la$ and the set of bars of length $ks'$ in $\quo_{\bar{g}}(\la)$. Hence $\la$ is an $\s$-core if and only if $\la_0$ is an $\bar{s}'$-core and $\la_1, \, \ldots , \, \la_{\frac{g-1}{2}}$ are $s'$-cores. Similarly, $\la$ is a $\bar{t}$-core if and only if $\la_0$ is a $\bar{t}'$-core and $\la_1, \, \ldots , \, \la_{\frac{g-1}{2}}$ are $t'$-cores. The result follows.
\end{proof}

\begin{corollary}
\label{genfunstbarcore} 
With the above notation, we have
$$\Psi_{\s,\bar{t}}(x)= \Psi_{\bar{s}',\bar{t}'}(x^g)  \Psi_{s',t'}(x^g)^{\frac{g-1}{2}}  F_{\bar{g}}(x),$$
where $F_{\g}(x)$
is the generating function for the number of $\g$-core partitions.
\end{corollary}

\begin{proof}
Each $(\bar{s},\bar{t})$-core partition $\la$ of $n$ is completely and uniquely determined by its $\g$-core ${\mathcal Cor}_{\g}(\la)$ and $\g$-quotient ${\mathcal Quo}_{\bar{g}}(\la)$. And, by Theorem \ref{stbarcore}, if $\la$ has $\g$-weight $w$, then ${\mathcal Quo}_{\g}(\la)$ can be any $\g$-quotient $(\la_0, \, \la_1, \, \ldots , \, \la_{\frac{g-1}{2}})$ of weight $w$ such that $\la_0$ is an $(\bar{s}', \bar{t}')$-core and $\la_1, \, \ldots , \, \la_{\frac{g-1}{2}}$ are $(s',t')$-cores. Writing $Q^{(\bar{s}',\bar{t}')}_{\bar{g}}(w)$ for the number of such $\g$-quotients, this shows that the enumerating function $\psi_{\s,\bar{t}}(n)$ of $(\bar{s},\bar{t})$-cores of $n$ is
\begin{equation}\label{ok2}
\psi_{\s,\bar{t}}(n)=\dis \sum_{w\geq 0} Q^{(\bar{s}',\bar{t}')}_{\bar{g}}(w)f_{\g}(n-gw).
\end{equation}
Viewed as a generating function,this transforms into
$$\Psi_{\s,\bar{t}}(x)= \Psi_{\bar{s'},\bar{t}'}(x^g)  \Psi_{s',t'}(x^g)^{\frac{g-1}{2}}  F_{\bar{g}}(x),$$
as claimed. 
\end{proof}

We will return to these functions in Section 3 and Section 5.

\section{$t$-cores that are not $g$-cores.}

In this section we let $g>1$, and consider first $t$-cores and self-conjugate $t$-cores that are {\it not} $g$-cores, and then $\bar{t}$-cores that are {\it not} $\bar{g}$-cores. We call these $t \backslash g${\it-cores}, self-conjugate $t \backslash g${\it-cores} and $\bar{t} \backslash \bar{g}${\it-cores} respectively. We first study $(s,t)$-core partitions that are {\it not} $g$-cores and $(\bar{s}, \bar{t})$-cores that are {\it not} $\bar{g}$-cores. Note we fix positive integers $s$ and $t$ where $s=s'g$ and $t=t'g$, and $g=$gcd$(s,t)$, where $s'>1$ and $t'>1$.

\subsection{$t\backslash g$-cores}
We begin with a succinct proof of Theorem \ref{infinitelymanystcores}. 

\begin{proof}[{\bf{Proof of Theorem \ref{infinitelymanystcores}}}]
Recall that there are infinitely many $g$-cores if $g>1$. For any $g$-core $\ga$, consider the (completely and uniquely defined) partition $\la$ with ${\mathcal Cor}_g(\la)=\gamma$ and ${\mathcal Quo}_g(\la)=(\la_1, \, \ldots , \, \la_g)=( (1), \, \emptyset, \, \ldots , \, \emptyset)$. Then $\la_1, \, \ldots , \, \la_g$ are certainly $(s',t')$-cores (since neither $s'$ nor $t'$ is 1), so that, by Theorem \ref{stcore}, $\la$ is an $(s'g,t'g)$-core. [Note that, in order to apply Theorem \ref{stcore}, we do not actually need the extra hypothesis that $g=\text{gcd}(s'g,t'g)$.] Also, since ${\mathcal Quo}_g(\la) \neq (  \emptyset, \, \ldots , \, \emptyset)$, the partition $\la$ is certainly {\it not} a $g$-core. Allowing $\ga$ to vary produces an infinite number of (distinct) $(s'g,t'g)$-cores which are not $g$-cores.
\end{proof}

Theorem \ref{yeahtnog} shows that, if $g \geq 4$ and $t'>1$, then, for any integer $n \geq g$, there exists a $t'g$-core of $n$ that is not a $g$-core. Our next theorem improves this result, particularly if $t' \geq 4$.

\begin{theorem} \label{f} Fix integers $g\geq 4$ and $t'>1$, and let $t=t'g$. Then, for every $n\geq g$, the number $\psi_{t \backslash g}(n)$ of $t \backslash g$-cores of $n$ is bounded below by $\dis \sum_{w= 1}^{\left\lfloor \frac{n}{g} \right\rfloor} Q^{t'}_g(w)$, where $Q^{t'}_g(w)$ is the number of $g$-quotients of weight $w$ all of whose components are $t'$-cores.

In particular, $\psi_{t \backslash g}(n) \geq g$.
If, furthermore, $t' \geq 4$, then $\psi_{t \backslash g}(n) \geq g \left\lfloor \frac{n}{g} \right\rfloor$.
\end{theorem}

\begin{proof}

First note that a partition $\la$ of $n$ is not a $g$-core if and only if it has $g$-weight $w \geq 1$. We therefore have (as in the proof of Theorem \ref{stcore})
$$
\psi_{t \backslash g}(n)  =  \dis \sum_{w \geq 1} Q^{t'}_g(w) f_g(n-wg)  =  \dis \sum_{w = 1}^{\left\lfloor \frac{n}{g} \right\rfloor} Q^{t'}_g(w) f_g(n-wg) \geq  \dis \sum_{w = 1}^{\left\lfloor \frac{n}{g} \right\rfloor} Q^{t'}_g(w), 
$$
the last inequality holding by Theorem \ref{Ono} since $g \geq 4$.

Note that, for any $w \geq 1$, we have $Q^{t'}_g(w) \geq g f_{t'}(w)$, as one can get $g$ different $g$-tuples of $t'$-cores of weight $w$ by placing a $t'$-core of $w$ (if any exists) in any of the $g$ positions available. We therefore get
$$\psi_{t \backslash g}(n)  \geq \dis \sum_{w = 1}^{\left\lfloor \frac{n}{g} \right\rfloor} g f_{t'}(w).$$
Now, since $t'>1$, the partition $(1)$ is the only $t'$-core of 1, whence $f_{t'}(1)=1$ (and, in fact, $Q^{t'}_g(1)=g$), so that $\psi_{t \backslash g}(n)  \geq g f_{t'}(1)=g$, as claimed.

If, furthermore, $t' \geq 4$, then, by Theorem \ref{Ono}, $f_{t'}(w) \geq 1$ for all $w$, and thus, in this case,
$$\psi_{t \backslash g}(n)  \geq \dis \sum_{w = 1}^{\left\lfloor \frac{n}{g} \right\rfloor} g = g \left\lfloor \frac{n}{g} \right\rfloor.$$

\end{proof}

\subsection{self-conjugate $t\backslash g$-cores.}

We have the following self-conjugate analogues to Theorem \ref{infinitelymanystcores} and Theorem \ref{f}.
\begin{theorem}
If $s,t,g>1$, then there are infinitely many self-conjugate $(s,t)$-core partitions that are not $g$-core partitions.
\end{theorem}

\begin{proof} Recall that there are infinitely many self-conjugate $g$-cores for $g>1$. For any self-conjugate $g$-core $\ga^*$, consider the (completely and uniquely defined) partition $\la$ with ${\mathcal Cor}_g({\la})=\ga^*$ and ${\mathcal Quo}_g(\la)=( (1), \, \emptyset, \, \ldots , \, \emptyset , \, (1))$. Then note that the partitions $(1)$ is an $(s',t')$-core since $s'>1$ and $t'>1$. Thus $\la$ is a self-conjugate $(s'g,t'g)$-core that is not a $g$-core. Allowing $\ga^*$ to vary completes the proof.
\end{proof}

\begin{theorem}
Fix integers $g=8$ or $g\geq10$ and $t'>1$, and let $t=t'g$. Then the number $\psi^*_{t\backslash g}(n)$ of self-conjugate $t$-cores of $n$ which are not $g$-cores satisfies
\begin{enumerate}
\item $\psi^*_{t\backslash g}(n) \geq \dis \sum_{w=1}^{\left\lfloor \frac{n}{2g} \right\rfloor} Q^{t'}_{\frac{g}{2}}(w)$ if $g$ is even, and\\
\item $\psi_{t\backslash g}(n) \geq \dis \sum_{2w_1+w_2=1}^{\left\lfloor \frac{n}{g} \right\rfloor} Q^{t'}_{\frac{g-1}{2}}(w_1)f^*_{t'}(w_2)$ if $g$ is odd.
\end{enumerate}
In particular, if $g$ is even and $n \geq 2g$, or if $g$ is odd and $n \geq g$, then $\psi_{t^*\backslash g}(n) >0$ i.e there exists a self-conjugate $t$-core of $n$ which is not a $g$-core. In addition, if $n \geq 2g$, then $\psi^*_{t\backslash g}(n)\geq \frac{g}{2}$.

If, furthermore, $g$ is even and $t'\geq 4$, then $\psi^*_{t\backslash g}(n)\geq \frac{g}{2} \left\lfloor \frac{n}{2g} \right\rfloor$. If $g$ is odd and $t'=8$ or $t' \geq 10$, then $\psi^*_{t\backslash g}(n)\geq 1+ \frac{g-1}{2} \left\lfloor \frac{n}{2g} \right\rfloor$.
\end{theorem} 

\begin{proof}

Again, note that a partition $\la$ of $n$ is not a $g$-core if and only if it has $g$-weight $w \geq 1$. Following the proof of Theorem \ref{genfunselfstcore}, we see that the number $\psi^*_{t\backslash g}(n)$ of self-conjugate $t$-cores of $n$ which are not $g$ cores is given by
$$\psi^*_{t\backslash g}(n) = \dis \sum_{w \geq 1} Q^{t'}_{\frac{g}{2}}(w) f^*_g(n-2wg) \;  \; \mbox{if $g$ is even},$$
and
$$\psi^*_{t\backslash g}(n) = \dis \sum_{2w_1+w_2\geq1} Q^{t'}_{\frac{g-1}{2}}(w_1)f^*_{t'}(w_2) f^*_g(n-(2w_1+w_2)g) \; \; \mbox{if $g$ is odd}.$$
Since $g=8$ or $g \geq 10$, we get, by Theorem \ref{fordsze},
$$\psi^*_{t\backslash g}(n) = \dis \sum_{w = 1}^{\left\lfloor \frac{n}{2g} \right\rfloor} Q^{t'}_{\frac{g}{2}}(w) f^*_g(n-2wg) \geq \sum_{w = 1}^{\left\lfloor \frac{n}{2g} \right\rfloor} Q^{t'}_{\frac{g}{2}}(w)  \; \; \mbox{if $g$ is even},$$
while
$$\psi^*_{t\backslash g}(n) \geq \dis \sum_{2w_1+w_2 = 1}^{\left\lfloor \frac{n}{g} \right\rfloor} Q^{t'}_{\frac{g-1}{2}}(w_1)f^*_{t'}(w_2)  \;  \; \mbox{if $g$ is odd}.$$
Suppose first that $g$ is even. If $\left\lfloor \frac{n}{2g} \right\rfloor =0$, i.e. if $n < 2g$, then $\psi^*_{t\backslash g}(n)=0$, since $n-2wg\leq 0$ for each $w$ in the sum above. If, on the other hand, $n \geq 2g$, then
$$\psi^*_{t\backslash g}(n) \geq \sum_{w = 1}^{\left\lfloor \frac{n}{2g} \right\rfloor} Q^{t'}_{\frac{g}{2}}(w) \geq \sum_{w = 1}^{\left\lfloor \frac{n}{2g} \right\rfloor} \frac{g}{2} f_{t'}(w).$$
Since $t'>1$, we have $f_{t'}(1)=1$ (and $Q^{t'}_{\frac{g}{2}}(1)= \frac{g}{2}$), so that $$\psi^*_{t\backslash g}(n) \geq \frac{g}{2} f_{t'}(1)=\frac{g}{2},$$ as claimed. In particular, there exists a self-conjugate $t$-core of $n$ which is not a $g$-core.

If, furthermore, $t'\geq 4$, then, by Theorem \ref{Ono}, $f_{t'}(w) \geq 1$ for all $w$, and thus, in this case, $$\psi^*_{t\backslash g}(n) \geq   \sum_{w = 1}^{\left\lfloor \frac{n}{2g} \right\rfloor} \frac{g}{2} = \frac{g}{2} \left\lfloor \frac{n}{2g} \right\rfloor,$$ as claimed.

Suppose now that $g$ is odd. For any $w_1 \geq 1$, we have $Q^{t'}_{\frac{g-1}{2}}(w_1) \geq \frac{g-1}{2} f_{t'}(w_1)$ while, for $w_1=0$, we have $Q^{t'}_{\frac{g-1}{2}}(w_1)=1$. Separating the sum according to $w_1=0$ and $w_1 \geq 1$, we thus obtain
$$\begin{array}{rcl} \psi^*_{t\backslash g}(n) & \geq & \dis \sum_{w_2 = 1}^{\left\lfloor \frac{n}{g} \right\rfloor} Q^{t'}_{\frac{g-1}{2}}(0)f^*_{t'}(w_2) +       \dis \sum_{w_1 \geq 1 \atop 2w_1+w_2 = 1 }^{\left\lfloor \frac{n}{g} \right\rfloor}\frac{g-1}{2} f_{t'}(w_1) f^*_{t'}(w_2) \\
 & = & \dis \sum_{w_2 = 1}^{\left\lfloor \frac{n}{g} \right\rfloor} f^*_{t'}(w_2) +       \dis \sum_{w_1 \geq 1 \atop 2w_1+w_2 = 1 }^{\left\lfloor \frac{n}{g} \right\rfloor}\frac{g-1}{2} f_{t'}(w_1) f^*_{t'}(w_2) .\end{array}$$
 
Since $t'>1$, we have $f^*_{t'}(1)=1$ (corresponding to the self-conjugate $t'$-core $(1)$). In particular, if $n \geq g$, then $\psi^*_{t\backslash g}(n) \geq 1$, and there exists a self-conjugate $t$-core of $n$ which is not a $g$-core.

If $n <2g$, then this is all we can say, as the second sum is empty. If, on the other hand, $n \geq 2g$, then, for $w_1=1$ and $w_2=0$ (which does contribute to the second sum), we get $f_{t'}(w_1)=f_{t'}(1)=1$ and $f^*_{t'}(w_2)=f^*_{t'}(0)=1$ (corresponding to the empty partition). Hence, in this case, we have $\psi^*_{t\backslash g}(n) \geq 1 + \frac{g-1}{2}= \frac{g+1}{2} \geq \frac{g}{2}$, as claimed.

If, furthermore, $t'=8$ or $t'\geq 10$, then, by Theorem \ref{Ono} and Theorem \ref{fordsze}, $f_{t'}(w_1) \geq 1$ and $f^*_{t'}(w_2) \geq 1$ for all $w_1$ and $w_2$, and thus, in this case
$$\psi^*_{t\backslash g}(n) \geq 1 + \dis \sum_{w_1 \geq 1 \atop 2w_1+w_2 = 1 }^{\left\lfloor \frac{n}{g} \right\rfloor}\frac{g-1}{2} \geq 1 + \dis \sum_{w_1 = 1 }^{\left\lfloor \frac{n}{2g} \right\rfloor}\frac{g-1}{2} = 1 + \frac{g-1}{2} \left\lfloor \frac{n}{2g} \right\rfloor,$$
as claimed.
\end{proof}

\subsection{$\bar{t}\backslash \bar{g}$-cores.}
We have the following bar-analogues of Theorem \ref{infinitelymanystcores} and Theorem \ref{f}.

\begin{theorem}\label{infinitelymanystbarcores}
Let $s',t',g$ be odd, nontrivial integers. Then there are infinitely many $(\overline{s'g},\overline{t'g})$-cores which are not $\g$-cores.
\end{theorem}
\begin{proof}
There are infinitely many $\bar{g}$-cores  for $g>1$. For any $\g$-core $\ga$, consider the (completely and uniquely defined) bar-partition $\la$ with ${\mathcal Cor}_g(\la)=\ga$ and ${\mathcal Quo}_g(\la)=(\la_0, \, \la_1, \, \ldots , \, \la_{\frac{g-1}{2}})=( (1), \, \emptyset, \, \ldots , \, \emptyset)$. Then $\la_0$ is an $(\bar{s'},\bar{t'})$-core (since $s',t'\neq1$), and $\la_1, \, \ldots , \, \la_{\frac{g-1}{2}}$ are certainly $(s',t')$-cores. Thus, by Theorem \ref{stbarcore}, $\la$ is a $(\overline{s'g},\overline{t'g})$-core (note here that Theorem \ref{stbarcore} does not require $g=\text{gcd}(s'g,t'g)$). Also, since ${\mathcal Quo}_{\bar{g}}(\la) \neq (  \emptyset, \, \ldots , \, \emptyset)$, the bar-partition $\la$ is certainly not a $\g$-core. Allowing $\ga$ to vary produces an infinite number of (distinct) $(\overline{s'g},\overline{t'g})$-cores which are not $\bar{g}$-cores.
\end{proof}

\nin
For odd integers $g$ and $t'$, let $Q^{\bar{t}'}_g(w)$ be the number of $\g$-quotients $(\la_0, \la_1, \ldots , \la_{\frac{g-1}{2}})$ of weight $w$ such that $\la_0$ is a $\bar{t}'$-core and $\la_1, \, \ldots , \, \la_{\frac{g-1}{2}}$ are $t'$-cores.

\begin{theorem} \label{tbarnotg} Let $g \geq 7$ and $t'>1$ be odd integers, and let $t=t'g$. Then, for any $n \geq g$, the number $\psi_{\bar{t}\backslash {\bar{g}}}(n)$ of $\bar{t}$-core partitions of $n$ which are not $\bar{g}$-cores satisfies
$$\psi_{\bar{t}\backslash {\bar{g}}}(n) \geq \dis \sum_{w= 1}^{\left\lfloor \frac{n}{g} \right\rfloor} Q^{\bar{t}'}_g(w).$$
In particular, $\psi_{\bar{t}\backslash {\bar{g}}}(n) \geq \frac{g+1}{2}$ and there exists a $\bar{t}$-core partition of $n$ which is not a $\g$-core. 
\end{theorem}

\begin{proof}
Note that a bar-partition $\la$ is a $\g$-core if and only if it has $\g$-weight 0. We thus have
$$\psi_{\bar{t}\backslash {\bar{g}}}(n)  =  \dis \sum_{w\geq 1} Q^{\bar{t}'}_g(w) f_{\g}(n-gw) =  \dis \sum_{w= 1}^{\left\lfloor \frac{n}{g} \right\rfloor} Q^{\bar{t}'}_g(w) f_{\g}(n-gw)  \geq  \dis \sum_{w= 1}^{\left\lfloor \frac{n}{g} \right\rfloor} Q^{\bar{t}'}_g(w)$$
(since, by Theorem \ref{kimingtime}, $f_{\g}(n-gw) \geq 1$ for all $1 \leq w \leq \left\lfloor \frac{n}{g} \right\rfloor$).

In particular, $\psi_{\bar{t}\backslash {\bar{g}}}(n) \geq Q^{\bar{t}'}_g(1)$, and $Q^{\bar{t}'}_g(1)=\frac{g+1}{2}$ since $t'>1$, so that $(1)$ is both the only $\bar{t}'$-core of $1$ and the only $t'$-core of 1. The result follows immediately.
\end{proof}

\section{Bijections}
Our main result in this section is Theorem \ref{Yyang}: a bijection between $(\bar{s},\bar{t})$-core partitions and $(s,t)^*$-core partitions when $g\geq 1$ and $s,t>1$ are odd. In Section 4.1, we reprove a recent bijection of J. Yang between $\bar{t}$-core partitions and self-conjugate $t$-core partitions for odd $t$ (Theorem \ref{Yangg}). In Section 4.2 we show that when $s,t>1$ are odd and $g=1$, then self-conjugate $(s,t)$-core partitions and $(\bar{s},\bar{t})$-core partitions are in bijection. These results, combined with results from Section 2.2 and Section 2.3, give, are the tools we use to construct our bijection in Section 4.3.

[We note that J. Wang and J. Yang \cite{WY} have recently extended the Yin-Yang diagram of Bessenrodt and Olsson to the case when $s$ is even and $t$ is odd. We do not consider this case here, although the calculations will be similar.]

\subsection{A bijection between $t^*$-cores and $\bar{t}$-cores}

The following result is Bijection 2 in Garvan, Kim and Stanton \cite{GKS}.

\begin{lemma}\label{JK}
For any integer $t \geq 1$, there is a correspondence $\varphi$ between the set of $t$-core partitions and $$\left \{(a_0, \, \ldots , \, a_{t-1}) \, | \, a_i \in \Z \;  0 \leq i \leq t-1 \; \mbox{and} \; \dis \sum_{i=0}^{t-1} a_i=0 \right \}.$$

Furthermore, if $\varphi(\la)=(a_0, \, \ldots , \, a_{t-1})$ for some $t$-core $\la$, then the conjugate partition $\la^{\vee}$ (which is also a $t$-core) satisfies $\varphi(\la^{\vee})=(-a_{t-1}, \, \ldots , \, -a_0)$.
\end{lemma}

[Note: Garvan, Kim and Stanton prove the existence of $\varphi$ using $t$-residue diagrams. It can also be derived from the study of so-called $t${\it -abacus configurations} for $t$-cores (see \cite[Chapter 2]{JK}).] 
The following follows immediately from the description of $\varphi(\la^{\vee})$ in Lemma \ref{JK}.

\begin{corollary}\label{bij}
Let $t>1$ be an odd integer. Then there is a correspondence between the set of self-conjugate $t$-cores and $$\left\{ (a_0, \, \ldots , \, a_{\frac{t-3}{2}}, \, 0 , \, -a_{\frac{t-3}{2}}, \, \ldots , \, -a_0) \; | \; a_i \in \Z \; \; , 0 \leq i \leq \frac{t-3}{2}\right \}.$$

\end{corollary}

Recall $\Delta(\lambda)=\{h_{ii}\}$ is the set of diagonal hook lengths of a self-conjugate partition $\lambda$. Ford, Mai, and Sze (Proposition 3, \cite{fms}) have another characterization of self-conjugate $t$-core partitions, in terms of $\Delta(\lambda)$, which we rewrite in the following way.
\begin{lemma} \label{fordzag} Let $\lambda$ be a $t$-core partition labeled by $\{ (a_0, \, \ldots , \, a_{t-1})\}$. Then $\lambda$ is self-conjugate if and only if for every $a_{\gamma}>0$:
\begin{enumerate}
\item $2(\gamma+\ell t)+1\in \Delta(\lambda)$ for all $0\leq\ell< a_{\gamma}$, and
\item If $h\in\Delta(\lambda)$ and $h+h'\equiv 0\pmod{2t}$ then $h'\not\in \Delta(\lambda)$.
\end{enumerate} 
\end{lemma}
In particular Lemma \ref{fordzag} (1) tells us how to recover $\Delta(\lambda)$, the set of diagonal hooks of a self-conjugate $t$-core partition, from its labeling $t$-tuple. Olsson has given the following $(\frac{t-1}{2})$-tuple characterization of $\bar{t}$-core partitions (Proposition (4.1) and Proposition (4.2), \cite{O3}).
\begin{lemma} \label{oldy} Let $t$ be an odd integer. Then the $\bar{t}$-core partitions can be encoded as $(\frac{t-1}{2})$-tuples: $(b'_1,b'_2,\cdots,b'_{\frac{t-1}{2}})$, where $b'_{i}\in {\mathbb Z}$. In particular, the set of parts of $\lambda$ can be recovered in the following way:
\begin{enumerate}
\item If $b'_i>0$ then $i+\ell t$ is a part of $\lambda$ for $0\leq \ell< b'_i$.\\
\item If $b'_i<0$ then $(t-i)+\ell t$ is a part of $\lambda$ for $0\leq \ell< |b'_i|$.
\end{enumerate}
\end{lemma}
We can now give a succinct proof of Theorem \ref{Yangg}.

\begin{proof}[{\bf{Proof of Theorem \ref{Yangg}}}]
By Corollary \ref{bij}, if $t \geq 1$ is odd, then there is a correspondence between self-conjugate $t$-cores and $t$-tuples $(a_0, a_1\ldots,0 \ldots ,-a_1,-a_0)$. By Lemma \ref{oldy}, $\bar{t}$-cores are similarly labeled by $(\frac{t-1}{2})$-tuples of integers $(b'_1,\cdots,b'_{\frac{t-1}{2}})$. Now let $\zeta$ be the map that sends $a_i$ to $b'_{i+1}$ for $0\leq i\leq \frac{t-3}{2}.$ 
\end{proof}

\begin{example}\label{GarveyMarcus} Let $t=3$. Then the self-conjugate $3$-core partition labeled by $(a_0)=(2,0,-2)$ is in bijection with the $\bar{3}$-core partition labeled by $(b_1)=(2)$. In particular the self-conjugate $3$-core $\lambda^*=(4,2,1,1)$ such that $\Delta(\lambda^*)=\{7,1\}$ is mapped to the $\bar{3}$-core $\bar{\lambda}=(4,1)$.
\end{example}
We note that the map in our proof of Theorem \ref{Yangg} does not work when $t$ is even, as the tuples that label each self-conjugate $t$-core partition are associated with an infinite family of $\bar{t}$-core partitions. 

\subsection{A bijection between $(s,t)^*$-core and $(\bar{s},\bar{t})$-cores when $g=1$}
In this section we show that, when $s,t>1$ are odd and $g=1$, there is a bijection between self-conjugate $(s,t)$-core partitions and $(\bar{s},\bar{t})$-core partitions via their corresponding lattice paths. First we detail the lattice constructions mentioned in Section $1.4$.

Ford, Mai, and Sze, in (Section 4, \cite{fms}), construct the diagonal hooks diagram, a $\left\lfloor\frac{s}{2}\right\rfloor \times \left\lfloor\frac{t}{2}\right\rfloor$ lattice of odd numbers, which we denote by $\mathcal{DH}_{s,t}$, as follows. First place $st-s-t$ in the upper left-hand corner position, labelled (1,1), in matrix notation. Then, applying Lemma \ref{fordzag} (1), we can deduct $2t$ for every move downwards, and $2s$ for every move rightward. The position $(i,j)$ will be filled with $st-s(2j-1)-t(2i-1)$ for $1\leq i\leq \left\lfloor\frac{s}{2}\right\rfloor$ and $1\leq j\leq \left\lfloor\frac{t}{2}\right\rfloor$. Some values in the diagram of diagonals will be negative. We separate the positive values in the upper left quadrant (the \emph{positive} side) from negative values in the bottom right portion (the \emph{negative} side) of the $\left\lfloor\frac{s}{2}\right\rfloor \times \left\lfloor\frac{t}{2}\right\rfloor$ lattice by a monotonic path which starts at the bottom left-most corner of the diagonal hooks diagram and ends at its right-most top corner. We will call this {\it the $\pm$border}. Then we let $|\{\mathcal{DH}_{s,t}\}|$ be the set of absolute values of all entries in the diagonal hooks diagram, and we have the following.
\begin{lemma}
Let $s>1$ and $t>1$ be such that $g=\gcd(s,t)=1$. Then the diagonal hook lengths of any self-conjugate $(s,t)$-core partition is a subset of $|\{\mathcal{DH}_{s,t}\}|$.  
\end{lemma}
We can now use $\mathcal{DH}_{s,t}$ and the $\pm$border to describe the following bijection, which is Lemma 7 in \cite{fms}.
\begin{lemma}\label{11} Let $s,t>1$ be such that $g=1$. The self-conjugate $(s,t)$-core partitions are in correspondence with monotonic paths in $\mathcal{DH}_{s,t}$. In particular, if $\pi$ is such a monotonic path and $\{d_1, \, \ldots , \, d_k\}$ is the set of values trapped between $\pi$ and the $\pm$border, then $\Delta(\lambda_{\pi})=\{|d_1|,\cdots, |d_k|\}$ for the corresponding self-conjugate $(s,t)$-core partition 
$\lambda_{\pi}$.
\end{lemma}
In other words, firstly, for a given self-conjugate $(s,t)$-core partition $\lambda$, a value on the positive side of row or column of the diagonal hooks diagram can occur as a diagonal hook length of $\lambda$ if and only if the absolute values in the same row or column but on the negative side {\it do not}. Secondly, if $|m|$ is a diagonal hook length of $\lambda$, and $m$ appears in a column or row of the positive (negative) side of the $\pm$border, then the absolute value of all entries in the same column or row as $m$ leading up to the $\pm$border also appear. [This corresponds to condition (2) in Lemma \ref{fordzag} when applied to both $s$ and $t$.] 
\begin{example}\label{611star} Let $s=7$ and $t=11$. Then the $3\times5$ lattice below is $\mathcal{DH}_{7,11}$. The $\pm$border is indicated by the dashed path. The self-conjugate $(7,11)$-core $\lambda^*$ corresponding to the lattice path $\pi$ (indicated by the solid black line) has diagonal hook set $$\Delta(\lambda^*)=\{5,3,1\}.$$ This is the partition $\lambda^*=(3^3)$.
\end{example}
\begin{center}
\scalebox{0.25}{\begin{tikzpicture}
\draw (0cm,4cm) node [rectangle, minimum size=2cm, inner sep=0pt, draw, anchor=south west] {\Huge$59$};
\draw (2cm,4cm) node [rectangle, minimum size=2cm, inner sep=0pt, draw, anchor=south west] {\Huge$45$};
\draw (4cm,4cm) node [rectangle, minimum size=2cm, inner sep=0pt, draw, anchor=south west] {\Huge$31$};
\draw (6cm,4cm) node [rectangle, minimum size=2cm, inner sep=0pt, draw, anchor=south west] {\Huge$17$};
\draw (8cm,4cm) node [rectangle, minimum size=2cm, inner sep=0pt, draw, anchor=south west] {\Huge$3$};

\draw (0cm,2cm) node [rectangle, minimum size=2cm, inner sep=0pt, draw, anchor=south west] {\Huge$37$};
\draw (2cm,2cm) node [rectangle, minimum size=2cm, inner sep=0pt, draw, anchor=south west] {\Huge$23$};
\draw (4cm,2cm) node [rectangle, minimum size=2cm, inner sep=0pt, draw, anchor=south west] {\Huge$9$};
\draw (6cm,2cm) node [rectangle, minimum size=2cm, inner sep=0pt, draw, anchor=south west] {\Huge$-5$};
\draw (8cm,2cm) node [rectangle, minimum size=2cm, inner sep=0pt, draw, anchor=south west] {\Huge$-19$};

\draw (0cm,0cm) node [rectangle, minimum size=2cm, inner sep=0pt, draw, anchor=south west] {\Huge$15$};
\draw (2cm,0cm) node [rectangle, minimum size=2cm, inner sep=0pt, draw, anchor=south west] {\Huge$1$};
\draw (4cm,0cm) node [rectangle, minimum size=2cm, inner sep=0pt, draw, anchor=south west] {\Huge$-13$};
\draw (6cm,0cm) node [rectangle, minimum size=2cm, inner sep=0pt, draw, anchor=south west] {\Huge$-27$};
\draw (8cm,0cm) node [rectangle, minimum size=2cm, inner sep=0pt, draw, anchor=south west] {\Huge$-41$};

\draw[line width=8pt](0,0)--(2,0)--(2,2)--(8,2)--(8,6)--(10,6);

\draw[line width=10pt, dashed](0,0)--(4,0)--(4,2)--(6,2)--(6,4)--(10,4)--(10,6);

\filldraw[black] (0,0) circle (0.25cm);
\filldraw[black] (10,6) circle (0.25cm);

\end{tikzpicture}}
\end{center}

Let $s,t>1$ be odd and $g=1$. Bessenrodt and Olsson \cite{BO} consider all possible values for parts of a simultaneous $(\bar{s},\bar{t})$-core partition. In an analogous argument to Anderson's, such values will be of the form $st-s-t-(\alpha s-\beta t)$ for $\alpha\geq 0$ and $\beta\geq 0$. We can visualize this in a way similar to Anderson, with a partially ordered set beginning with $st-s-t$ in the upper rightmost corner. Each move down reduces the value by $t$, each move to the right reduces the value by $s$. [We note that our poset orientation here is downward as opposed to Bessenrodt-Olsson's upward orientation; the reason for this will be made clear in the proof of Theorem \ref{DiagYang}.]

However, not all part-values contained in the poset are acceptable parts for a $(\bar{s},\bar{t})$-core partition. Consider $1\leq j\leq t$. By Olsson's characterization of $\bar{t}$-core partitions (Equation 5(ii), Chapter 4, \cite{O3}) either parts of residue $j$ modulo $t$ or $t-j$ modulo $t$ can appear, but not both. However, $st-s-t-(\frac{t-3}{2} s-\frac{s-3}{2} t)=\frac{t+s}{2}$, and $\frac{t+s}{2}-s=\frac{t-s}{2}$, and the inclusion of both values would violate this condition modulo $s$. Hence a rectangular section of $\frac{s-1}{2}\times\frac{t-1}{2}$ values in the poset, those {\it above} (in the sense of the partial order) and {\it including} $\frac{t+2}{2}$ must be eliminated. What remains are two sets of numbers: the {\it Yin} portion of the diagram, below the eliminated rectangular section, and the {\it Yang} portion of the diagram, to the right of the rectangular section. Then Bessenrodt and Olsson (Remark 4.1, \cite{BO}) show that by rotating the Yang portion of the diagram $180^{\circ}$, one can attach the Yang portion to the Yin portion to produce a $\frac{s-1}{2}\times\frac{t-1}{2}$ lattice with value $t(\frac{s-1}{2})-s$ in the top leftmost corner, $(t-s)$ in the bottom leftmost corner, $s(\frac{t-1}{2})-t$ in the bottom rightmost corner, and $(\frac{t-s}{2})$ in the top rightmost corner.

The following result appears as \cite[Remark 3.1]{BO}.

\begin{lemma} Let $s,t>1$ be odd and $g=1$. Then the Yin-Yang diagram contains all possible values of parts for an $(\bar{s},\bar{t})$-core partition.
\end{lemma}

In particular, values from a column or row of the Yin diagram can be included as parts of an $(\bar{s},\bar{t})$-core partition if and only if no values from the same column or row but in the Yang diagram are included. Let the ${\mathcal Y}$-border be the path separating the Yin part from the Yang part. The following is a result of Bessenrodt and Olsson (Theorem 3.2, \cite{BO}).

\begin{lemma} \label{22} Let $s,t>1$ be such that $g=1$. Then $(\bar{s},\bar{t})$-core partitions are labeled by monotonic paths $\pi$ in the Yin-Yang diagram. In particular, $(\lambda_1,\cdots,\lambda_k)$ are the parts of $\lambda$ where $\lambda_i$ are the values trapped between $\pi$ and the ${\mathcal Y}$-border.
\end{lemma} 
[We note that recent results by C. Deng \cite{De} describe precisely which monotonic paths correspond to the so-called ``even" $(\bar{s},\bar{t})$-core partitions.]

\begin{example}\label{1627bar} Let $s=7$ and $t=11$. Then the Yin-Yang diagram is below, and the ${\mathcal Y}$-border between the two portions is indicated by a dashed line. The $(\bar{7},\overline{11})$-core partition $\bar{\lambda}$ that corresponds to the path below is $\bar{\lambda}=(6).$ [Here the elements of the Yang portion of the diagram are labeled with a negative sign (for clarity), so the absolute value of those values will need to taken.] 
\end{example}

\smallskip

\begin{center}
\scalebox{0.25}{\begin{tikzpicture}
 size=2cm, inner sep=0pt, draw, anchor=south west];

\draw (0cm, 4cm) node [rectangle, minimum size=2cm, inner sep=0pt, draw, anchor=south west] {\Huge$26$};
\draw (2cm, 4cm) node [rectangle, minimum size=2cm, inner sep=0pt, draw, anchor=south west] {\Huge$19$};
\draw (4cm, 4cm) node [rectangle, minimum size=2cm, inner sep=0pt, draw, anchor=south west] {\Huge$12$};
\draw (6cm, 4cm) node [rectangle, minimum size=2cm, inner sep=0pt, draw, anchor=south west] {\Huge$5$};
\draw (8cm, 4cm) node [rectangle, minimum size=2cm, inner sep=0pt, draw, anchor=south west] {\Huge$-2$};
\draw (0cm, 2cm) node [rectangle, minimum size=2cm, inner sep=0pt, draw, anchor=south west] {\Huge$15$};
\draw (2cm, 2cm) node [rectangle, minimum size=2cm, inner sep=0pt, draw, anchor=south west] {\Huge$8$};
\draw (4cm, 2cm) node [rectangle, minimum size=2cm, inner sep=0pt, draw, anchor=south west] {\Huge$1$};
\draw (6cm, 2cm) node [rectangle, minimum size=2cm, inner sep=0pt, draw, anchor=south west] {\Huge$-6$};
\draw (8cm, 2cm) node [rectangle, minimum size=2cm, inner sep=0pt, draw, anchor=south west] {\Huge$-13$};

\draw (0cm, 0cm) node [rectangle, minimum size=2cm, inner sep=0pt, draw, anchor=south west] {\Huge$4$};
\draw (2cm, 0cm) node [rectangle, minimum size=2cm, inner sep=0pt, draw, anchor=south west] {\Huge$-3$};
\draw (4cm, 0cm) node [rectangle, minimum size=2cm, inner sep=0pt, draw, anchor=south west] {\Huge$-10$};
\draw (6cm, 0cm) node [rectangle, minimum size=2cm, inner sep=0pt, draw, anchor=south west] {\Huge$-17$};
\draw (8cm, 0cm) node [rectangle, minimum size=2cm, inner sep=0pt, draw, anchor=south west] {\Huge$-24$};
\draw[line width=8pt](0,0)--(2,0)--(2,2)--(8,2)--(8,6)--(10,6);
\draw[line width=11pt, dashed](0,0)--(2,0)--(2,2)--(6,2)--(6,4)--(8,4)--(8,6)--(10,6);
\filldraw[black] (0,0) circle (0.25cm);
\filldraw[black] (10,6) circle (0.25cm);
\end{tikzpicture}}
\end{center}
We are now in a position to prove the main result of this subsection.

\begin{theorem}\label{DiagYang} Let $s,t>1$ be odd and coprime. Then self-conjugate $(s,t)$-core partitions are in bijection with $(\bar{s},\bar{t})$-core partitions.
\end{theorem}

\begin{proof} By Lemma \ref{11}, we know that the self-conjugate $(s,t)$-core partitions are labeled by monotonic paths in an $\left\lfloor\frac{s}{2}\right\rfloor \times \left\lfloor\frac{t}{2}\right\rfloor$ lattice (the diagonal hooks diagram). Consider the self-conjugate $(s,t)$-core $\lambda^*_{\pi}$ determined by the diagonal hook values trapped between a monotonic path $\pi$ and the $\pm$boundary. By Lemma \ref{22} we know that the $(\bar{s},\bar{t})$-core partitions are labeled by monotonic (as a consequence of our orientation) paths in a $\frac{s-1}{2}\times\frac{t-1}{2}$ lattice.
We delete the values of the diagonals hook diagram and replace them with the values of the Yin-Yang diagram, leaving the path $\pi$ unchanged. Then the corresponding $(\bar{s},\bar{t})$-core partition $\bar{\lambda}_{\pi}$ is the one described by the part-values trapped between $\pi$ and the ${\mathcal Y}$-border. The map works the same in the other direction.
\end{proof}
Let $\gamma_{s,t}$ be the bijection described in the proof of Theorem \ref{DiagYang}. 
\smallskip
\begin{example} \label{help} Consider the self-conjugate $(7,11)$-core partition $\lambda^*=(3^3)$ in Example \ref{611star}. Then $\gamma_{7,11}(\lambda^*)$ is the $(\bar{7},\overline{11})$-core partition $\bar{\lambda}=(6)$ in Example \ref{1627bar}.
\end{example}

\subsection{A bijection between $(s,t)^*$-cores and $(\bar{s},\bar{t})$-cores when $g>1$ and odd.}

Let $s,t>1$ be odd and $g>1$ be odd. Using bijections from Section 4.1 and Section 4.2, we construct a correspondence between self-conjugate $(s,t)$-core partitions and $(\bar{s},\bar{t})$-core partitions. In particular, we apply results from Section 2.2 and Section 2.3 and move from the $\bar{g}$-core and $\bar{g}$-quotient of an $(\bar{s},\bar{t})$-core to the $g$-core and $g$-quotient of a self-conjugate $(s,t)$-core.  

\begin{theorem}\label{Yyang} Let $s,t$ be odd and $g>1$. Then self-conjugate $(s,t)$-core partitions are in correspondence with $(\bar{s},\bar{t})$-core partitions.
\end{theorem}

\begin{proof}
By Theorem \ref{stcore}, each self-conjugate $(s,t)$-core partition $\lambda$ is labeled by a self-conjugate $g$-core ${\mathcal Cor}_g(\lambda)$ and $g$-quotient ${\mathcal Quo}_g(\lambda)=(\lambda_0,\cdots,\lambda_{g-1})$ where each $\lambda_i$ is an $(s',t')$-core for $0\leq i\leq \frac{g-3}{2}$ and $\lambda^{\vee}_i=\lambda_{g-i-1}.$ In particular, since $g$ is odd, $\lambda_{\frac{g-1}{2}}$ is a $(s',t')^*$-core. Consider the following map $\Gamma_{s,t}$.
\begin{enumerate}
\item $\Gamma_{s,t}$ sends the self-conjugate $g$-core ${\mathcal Cor}_g(\lambda)$ to its corresponding $\bar{g}$-core $\zeta({\mathcal Cor}_g(\lambda))$ via the bijection $\zeta$ of Theorem \ref{Yangg}.
\item  $\Gamma_{s,t}$ sends the $(s',t')$-core $\lambda_i$, for $0\leq i\leq \frac{g-3}{2}$, to itself, renaming it $\overline{\lambda}_{i+1}$.
\item  $\Gamma_{s,t}$ sends the self-conjugate $(s',t')$-core $\lambda_{\frac{g-1}{2}}$ to its corresponding $(\overline{s}',\overline{t}')$-core $\overline{\lambda}_{0}$, via the bijection $\gamma_{s,t}$ of Theorem \ref{DiagYang}.
\end{enumerate}
Then set ${\mathcal Cor}_{\bar{g}}(\overline{\lambda})=\zeta({\mathcal Cor}_g(\lambda))$ and $(\overline{\lambda}_0,\overline{\lambda}_1\cdots,\overline{\lambda}_{\frac{g-3}{2}})={\mathcal Quo}_g(\overline{\lambda})$ for some partition $\overline{\lambda}$. This, by Theorem \ref{stbarcore}, uniquely determines an $(\bar{s},\bar{t})$-core $\overline{\lambda}$. Since the map goes in both directions, we are done.
\end{proof}

\begin{example} Under $\Gamma_{21,33}$, the self-conjugate $(21,33)$-core partition $$\lambda=(21,20,12^4,11^2,10,9,8,6,2^8,1)$$ where ${\mathcal Cor}_{3}(\lambda)=(4,2,1,1)$ and ${\mathcal Quo}_{3}(\lambda)=((5,3^3,2^2,1^3),(3^3),(9,6,4,1,1))$ corresponds to a $(\overline{21},\overline{33})$-core partition $\overline{\lambda}$ where ${\mathcal Cor}_{\bar{3}}(\overline{\lambda})=(4,1)$ and ${\mathcal Quo}_{\bar{3}}(\overline{\lambda})=((6),(5,3^3,2^2,1))$, since the self-conjugate $(7,11)$-core $(3^3)$ corresponds to $(\bar{7}, \overline{11})$-core (6) by Example \ref{help}. In particular $\overline{\lambda}=(20,19,18,10,8,7,4)$.
\end{example}

The following corollary is immediate.
\begin{corollary} Let $t',g>1$ be odd. The self-conjugate $t'g$-core partitions that are not $g$-core partitions are in bijection with the $\overline{t'g}$-core partitions that are not $\bar{g}$-core partitions.
\end{corollary}
\begin{proof} This follows from Theorem \ref{Yyang} when $s=t$ and Theorem \ref{Yangg}.
\end{proof}

\section{Ramanujan-type congruences}

Srinivasa Ramanujan was the first to notice several remarkable arithmetic properties of the partition function $p(n)$.  

\label{ramanujan_congs}
In particular, he noted that, for all $k\geq 0$,
\begin{eqnarray*}
p(5k + 4) &=& 0\pmod{5}, \\
p(7k + 5) &=& 0\pmod{7}, \\
p(11k + 6) &=& 0\pmod{11}.
\end{eqnarray*}

In 1990, Garvan, Kim, and Stanton \cite{GKS} proved the above congruences using, among other things, $t$-core partitions. They also showed that $t$-core partitions and self-conjugate $t$-cores of $n$ satisfy Ramanujan-type congruences.
\begin{theorem}\label{garvey} Let $f_t(n)$ be the number of $t$-core partitions of $n$. Then, for all $k\geq 0,$
\begin{eqnarray*}
f_5(5k+4) &\equiv & 0 \pmod{5}, \\
f_7(7k+5) &\equiv & 0 \pmod{7}, \text{\ \ and}\\
f_{11}(11k+6) &\equiv & 0 \pmod{11}.
\end{eqnarray*} 
\end{theorem} 
[Note: Theorem \ref{garvey} is actually a special case of a much general result. In particular, Garvan, Kim and Stanton show (see \cite[Corollary 1]{GKS})
\begin{eqnarray*}
f_5(5^\alpha k -1) &\equiv & 0 \pmod{5^\alpha}, \\
f_7(7^\alpha k-2) &\equiv & 0 \pmod{7^\alpha}, \text{\ \ and}\\
f_{11}(11^\alpha k-5) &\equiv & 0 \pmod{11^\alpha}
\end{eqnarray*} 
for all $\alpha \geq 1$ and $k\geq 1$.]
 
Recently, the second and third author made the following observation (Theorem 4.1, \cite{NS16}) about properties of $\bar{t}$-core partitions.
\begin{theorem}\label{naths}
\label{pbar-cores-parity}
Suppose $f_{\bar{p}}(n)$ is the number of $\bar{p}$-core partitions of $n$, where $p\geq 5$ is prime and $r$ is such that $1\leq r\leq p-1$ and $24r+1$ is a quadratic nonresidue modulo $p$. Then, for all $k\geq 0,$ 
$$
f_{\bar{p}}(pk+r) \equiv 0\pmod{2}.
$$
\end{theorem}
The following result will aid us in using Theorem \ref{garvey} and Theorem \ref{naths} to find congruences of $(s,t)$-core partitions and $(\bar{s},\bar{t})$-core partitions for $g>1$.
\begin{theorem} 
\label{arithmetic_progression_values} 
Let 
$$
A(x) = \sum_{n\geq 0} a(n)x^n  \text{\ \ and \ \ } B(x) = \sum_{n\geq 0} b(n)x^n
$$
be two generating functions and let $C(x) = A(x^g)B(x)$ for some positive integer $g.$  Let $r$ be an integer satisfying $1\leq r\leq g-1.$  Then, for any $n\geq 0,$ 
$$
c(gk+r) = \sum_{m\geq 0} a(k-m)b(gm+r).
$$
\end{theorem}
where $C(x)=\sum_{n\leq0}c(n)x^n.$
\begin{proof}
This follows directly from the Cauchy product of $A(x^g)$ and $B(x).$  
\end{proof}

\begin{corollary}
\label{cor_congruences_arithmetic_progression}
Consider the generating functions $A(x), B(x),$ and $C(x)$ as defined in Theorem \ref{arithmetic_progression_values}.  If, for all $k\geq 0,$ $b(gk+r) \equiv 0 \pmod{M}$ for some integer $M,$ then for all $k\geq 0,$ $c(gk+r) \equiv 0 \pmod{M}.$ 
\end{corollary}
We use Corollary \ref{cor_congruences_arithmetic_progression} to identify congruences in arithmetic progressions satisfied by $\psi_{s,t}(n)$ (which counts the number of $(s,t)$-core partitions of $n$) and $\psi_{\s,\bar{t}}(n)$ (which counts the $(\bar{s}, \bar{t})$-core partitions of $n$) for specific families of values of $s$ and $t.$ 
We begin with $\psi_{s,t}(n).$ 
\begin{theorem}\label{zz1}
Let $s>1$ and $t>1$.  Then, for all $k\geq 0,$
\begin{eqnarray*}
\psi_{s,t}(5k+4) & \equiv & 0 \pmod{5}, \\
\psi_{s,t}(7k+5) &\equiv & 0 \pmod{7}, \text{\ \ and}\\
\psi_{s,t}(11k+6) &\equiv & 0 \pmod{11}.
\end{eqnarray*} 
\end{theorem}
\begin{proof}
Using the notation from Corollary \ref{genfunstcore}, let $F_g(n)$ be the number of $g$-core partitions of $n.$ The result then follows directly from Theorem \ref{garvey}, Corollary \ref{genfunstcore} and Corollary \ref{cor_congruences_arithmetic_progression}.
\end{proof}
Let $\psi_{t\backslash g}(n)$ be the number of $t$-cores of $n$ that are not $g$-cores. Then we have the following.
\begin{corollary} Let $t>1$ and $g>1$. Then, for all $k\geq 0$,
\begin{eqnarray*}
\psi_{t\backslash g}(5k+4) & \equiv & 0 \pmod{5}, \\
\psi_{t\backslash g}(7k+5) &\equiv & 0 \pmod{7}, \text{\ \ and}\\
\psi_{t\backslash g}(11k+6) &\equiv & 0 \pmod{11}.
\end{eqnarray*} 
\end{corollary}
\begin{proof} The number of $t$-cores of $n$ that are not $g$-cores is given by $$\psi_{t\backslash g}(n)=\sum_{w\geq 1}Q_{t'}(g,w)f_g(n-wg)$$ by Theorem \ref{f}. Since, when $n=gk+r$, each $f_g(n-wg)$ is of the form $f_g(gk+r-wg)=f_{g}(g(k-w)+r)$, the result then follows from Theorem \ref{garvey} and Theorem \ref{zz1} when $s=t$.
\end{proof}
From Corollary \ref{genfunstbarcore}, we see that 
$$\Psi_{\s,\bar{t}}(x)= \Psi_{\bar{s}',\bar{t}'}(x^g)  \Psi_{s',t'}(x^g)^{\frac{g-1}{2}}  F_{\bar{g}}(x),$$
where $F_{\g}(x)= \dis \prod_{n=1}^{\infty}\frac{(1-x^{2n})(1-x^{gn})^{\frac{g+1}{2}}}{(1-x^n)(1-x^{2gn})}$ is the generating function for the number of $\g$-core partitions.  Using the notation of Theorem \ref{arithmetic_progression_values}, we see that $\Psi_{\s,\bar{t}}(x)= A(x^g)B(x)$ where $A(x) = \Psi_{\bar{s}',\bar{t}'}(x^g)  \Psi_{s',t'}(x^g)^{\frac{g-1}{2}}$
and $B(x)= F_{\bar{g}}(x).$
In this context, it is easy to prove the following parity results satisfied by $\psi_{\s,\bar{t}}(n)$ for specific values of $s$ and $t.$  
\begin{theorem}\label{zz2}
Let $s>1$ and $t>1$ be such that $g\geq 5$ is a prime.  Let $r$ be an integer, $1\leq r\leq p-1,$ such that $24r+1$ is a quadratic nonresidue modulo $p.$  Then, for all $k\geq 0,$ 
$$\psi_{\bar{s},\bar{t}}(gk+r) \equiv 0\pmod{2}.$$
\end{theorem}
\begin{proof}
Theorem \ref{pbar-cores-parity} gives us the parity of the values $F_{\g}(gn+r)$. An application of Corollary \ref{cor_congruences_arithmetic_progression} completes the proof of this theorem.  
\end{proof}
\begin{corollary} Let $t=t'g$, where $g\geq 5$ is a prime. Let $r$ be an integer, $1\leq r\leq p-1,$ such that $24r+1$ is a quadratic nonresidue modulo $p.$ Then, for all $k\geq 0,$ 
$\psi_{\bar{t}\backslash\bar{g}}(gk+r) \equiv 0\pmod{2}.$
\end{corollary}
\begin{proof} Note that the result is obvious if $t'=1$, in which case $\psi_{\bar{t}\backslash\bar{g}}(gk+r)=0$ (as each $\bar{t}$-core is a $\bar{g}$-core). If, on the other hand, $t'>1$, then the number of partitions of $n\geq g$ which are $\bar{t}$-cores but not $\bar{g}$-cores is $\psi_{\bar{t}\backslash\bar{g}}(n)=\dis \sum_{w\geq 1} Q_{\bar{t'}}(g,w)F_{\g}(n-gw)$, by the proof of Theorem \ref{tbarnotg}. The result follows then from Theorem \ref{naths} and Theorem \ref{zz2} when $s=t$.
\end{proof}
\section{Index of functions}
Here we list the notation for enumerating functions, generating functions, and bijections that appear throughout the paper. Note any time $\bar{t}$ appears it indicates that $t$ is odd.\\
\addcontentsline{toc}{chapter}{List of Symbols}
\begin{tabular}{cp{0.8\textwidth}}
  \bf{enumerating function} & \bf{counts}\\\\
  $p(n)$ & integer partitions of $n$\\
  $f_t(n)$ & $t$-core partitions of $n$\\
  $f^*_t(n)$ & self-conjugate $t$-core partitions of $n$\\
  $f_{\bar{t}}(n)$ & $\bar{t}$-core partitions of $n$\\
  $Q_{t}(n)$ & $t$-quotients of $n$\\
  $Q_{\bar{t}}(n)$ & $\bar{t}$-quotients of $n$\\
  $Q^{(s',t')}_g(n)$ & $g$-quotients of $n$ made up of $(s',t')$-cores\\
  $Q^{(\bar{s}',\bar{t}')}_{\bar{g}}(n)$ & $\bar{g}$-quotients of $n$ made up of $(s',t')$-cores and a $(\bar{s}',\bar{t}')$-core\\
  $Q^{t}_g(n)$ & $g$-quotients of $n$ made up of $t$-cores\\
  $Q^{\bar{t}'}_{\bar{g}}(n)$ & $\bar{g}$-quotients of $n$ made up of $t'$-cores and a $\bar{t}'$-core\\
  $\psi_{s,t}(n)$ & $(s,t)$-core partitions of $n$\\
  $\psi^*_{s,t}(n)$ & self-conjugate $(s,t)$-core partitions of $n$\\
  $\psi_{\bar{s},\bar{t}}(n)$ & $(\bar{s},\bar{t})$-core partitions of $n$\\
  $\psi_{t\backslash g}(n)$ & $t$-core partitions of $n$ that are not $g$-core\\
  $\psi^*_{t\backslash g}(n)$ & self-conjugate $t$-core partitions of $n$ that are not $g$-core\\
  $\psi_{\bar{t}\backslash\bar{g}}(n)$ & $\bar{t}$-core partitions of $n$ that are not $\bar{g}$-core\\\\
  \bf{generating function} & \bf{coefficients are}\\\\
  $P(x)$ & $p(n)$\\
  $F_t(x)$ & $f_t(n)$\\
  $F_t^*(x)$ & $f^*_t(n)$\\
  $F_{\bar{t}}(x)$ &  $f_{\bar{t}}(n)$\\
  $\Psi_{s,t}(x)$ & $\psi(n)$ \\
  $\Psi^*_{s,t}(x)$ & $\psi^*_{s,t}(n)$\\
  $\Psi_{\bar{s},\bar{t}}(x)$ & $\psi^*_{\bar{s},\bar{t}}(n)$\\\\
    \bf{bijection} & \bf{between}\\\\
  $\zeta$ & the set of self-conjugate $t$-cores and $\bar{t}$-cores\\
  $\gamma_{s,t}$ & self-conjugate $(s,t)$-cores and $(\bar{s},\bar{t})$-cores with $s,t$ odd, $g=1$\\
  $\Gamma_{s,t}$ & self-conjugate $(s,t)$-cores and $(\bar{s},\bar{t})$-cores with $s,t$ odd, $g>1$\\\\
\end{tabular}\\\\
{\bf{Acknowledgements.}} Part of this work was done at the Centre Interfacultaire Bernoulli (CIB), in the \'{E}cole Polytechnique F\'{e}d\'{e}rale de Lausanne (Switzerland), during the Semester {\emph{Local Representation Theory and Simple Groups}}. The first two authors are grateful to the CIB for their financial and logistical support. The first author also acknowledges financial support from the Engineering and Physical Sciences Research Council grant \emph{Combinatorial Representation Theory} EP/M019292/1. The second author was supported by PSC-TRADA-46-493 and thanks George Andrews who supported a visit to Penn State where this research began. The second author also thanks Christopher R. H. Hanusa for helpful conversations on diagrams and references, and notes that some diagrams were made using the ytab package. All of the authors thank the anonymous referee for the careful reading and detailed and helpful suggestions.

\end{document}